\documentclass{amsart}
\usepackage{amsmath,amssymb,amsthm,amsfonts,amscd,mathrsfs,mathtools,enumerate}
\usepackage{tikz,graphicx}
\usepackage[all]{xy}
\usepackage[hyphens]{url}
\usepackage{hyperref}
\usepackage{caption}
\usetikzlibrary{decorations.markings}

\theoremstyle{plain}
    \newtheorem{thm}{Theorem}[section]
    \newtheorem{lem}[thm]   {Lemma}
    \newtheorem{cor}[thm]   {Corollary}
    \newtheorem{prop}[thm]  {Proposition}
    \newtheorem{question}[thm]{Question}
\theoremstyle{definition}
    \newtheorem{defn}[thm]  {Definition}
    
    \newtheorem{conj}[thm]{Conjecture}
    
    \newtheorem{ex}[thm]{Example}
    \newtheorem{rem}[thm]{Remark}

\def\A{{\mathcal A}}

\def\max{\mathrm{max}}

\def\Hom{\mathrm{Hom}}

\def\Ext{\mathrm{Ext}}
\def\Mod{\mathrm{Mod}}

\def\Ab{\mathrm{Ab}}

\def\cat{\mathsf{cat}}
\def\secat{\mathsf{secat}}

\def\dim{\mathrm{dim}}
\def\ker{\mathrm{ker}}

\def\res{\mathrm{res}}

\newcommand{\be}{\begin{enumerate}}
\newcommand{\ee}{\end{enumerate}}
\newcommand{\R}{\mathbb{R}}
\newcommand{\Z}{\mathbb{Z}}

\newcommand{\F}{\mathcal{F}}
\newcommand{\FIN}{\mathcal{FIN}}
\newcommand{\VCYC}{\mathcal{VCYC}}
\newcommand{\PP}{\mathcal{P}}
\newcommand{\FG}{\mathcal{G}}
\newcommand{\ul}{\underline}
\newcommand{\CC}{\mathcal{C}}

\newcommand{\nsub}{\mathrel{\unlhd}}

\newcommand{\cld}{{\sf cd}}
\newcommand{\gld}{{\sf gd}}
\newcommand{\catG}{{\sf cat}_G}
\newcommand{\secatG}{{\sf secat}_G}
\newcommand{\cldG}{{\sf cd}_\mathcal{G}}
\newcommand{\gldG}{{\sf gd}_\mathcal{G}}
\newcommand{\OG}{\mathcal{O}G}
\newcommand{\OF}{\mathcal{O}_\mathcal{F}}
\newcommand{\G}{\Gamma}
\newcommand{\GG}{\pi\rtimes G}

\newcommand{\OFG}{\mathcal{O}_\F\Gamma}
\newcommand{\OFGG}{\mathcal{O}_\FG (\pi\rtimes G)}

\begin{document}

\title[Equivariant dimensions]{Equivariant dimensions of groups with operators}

\author{Mark Grant}

\author{Ehud Meir}

\author{Irakli Patchkoria}

\address{Institute of Mathematics,
Fraser Noble Building,
University of Aberdeen,
Aberdeen AB24 3UE,
UK}

\email{mark.grant@abdn.ac.uk}

\email{ehud.meir@abdn.ac.uk}

\email{irakli.patchkoria@abdn.ac.uk}

\date{\today}


\begin{abstract} Let $\pi$ be a group equipped with an action of a second group $G$ by automorphisms. We define the equivariant cohomological dimension $\cld_G(\pi)$, the equivariant geometric dimension $\gld_G(\pi)$, and the equivariant Lusternik-Schnirelmann category $\cat_G(\pi)$ in terms of the Bredon dimensions and classifying space of the family of subgroups of the semi-direct product $\pi\rtimes G$ consisting of sub-conjugates of $G$. When $G$ is finite, we extend theorems of Eilenberg--Ganea and Stallings--Swan to the equivariant setting, thereby showing that all three invariants coincide (except for the possibility of a $G$-group $\pi$ with $\cat_G(\pi)=\cld_G(\pi)=2$ and $\gld_G(\pi)=3$). A main ingredient is the purely algebraic result that the cohomological dimension of any finite group with respect to any family of proper subgroups is greater than one. This implies a Stallings--Swan type result for families of subgroups which do not contain all finite subgroups.
\end{abstract}

%\thanks{}

\maketitle
\section{Introduction}\label{sec:intro}

The purpose of this article is to show that famous theorems of Eilenberg and Ganea \cite{EG} and Stallings \cite{Sta} and Swan \cite{Swa} relating three quantities associated to discrete groups---the geometric dimension, the cohomological dimension and the Lusternik--Schnirelmann category of a classifying space---admit equivariant generalisations to the setting of groups with operators.

Let $\pi$ be a discrete group. A connected CW-complex whose fundamental group is $\pi$ and whose higher homotopy groups are all trivial is called a $K(\pi,1)$-complex. Such a space is unique up to homotopy type. The \emph{geometric dimension} of $\pi$, denoted $\gld(\pi)$, is the minimal dimension of a $K(\pi,1)$ complex (alternatively, the minimal dimension of a contractible complex on which $\pi$ acts freely). The \emph{Lusternik--Schnirelmann (LS) category} of $\pi$ is $\cat(\pi):=\cat(K(\pi,1))$, the LS category of a $K(\pi,1)$ complex, which is well-defined by homotopy invariance. (Recall that the homotopy invariant $\cat(X)$ is defined to be the minimal integer $k$ for which there exists a cover of $X$ by open sets $U_0,\ldots , U_k$ such that each inclusion $U_i\hookrightarrow X$ is null-homotopic; see \cite{CLOT} for further details.) Finally, the \emph{cohomological dimension} of $\pi$, denoted $\cld(\pi)$, may be defined topologically as the minimal $d$ such that $H^{d+1}(K(\pi,1);\mathcal{M})=0$ for all local coefficient systems $\mathcal{M}$ on $K(\pi,1)$, or algebraically as the projective dimension of the trivial module $\Z$ in the category of $\pi$-modules. It is easy to check that if one of these invariants is zero then $\pi$ is the trivial group, and the other two invariants are zero as well. Note that all three invariants may be infinite; this happens for example if $\pi$ has torsion elements.

With these definitions, the theorems we will generalise are as follows. Eilenberg and Ganea showed in \cite{EG} that for any discrete group $\pi$ there is a chain of inequalities
\[
\cld(\pi)\leq \cat(\pi)\leq \gld(\pi)\leq \sup\{3,\cld(\pi)\},
\]
and furthermore if $\cld(\pi)=2$ then $\cat(\pi)=2$, and if $\cat(\pi)=1$ then $\pi$ is a free group and $\cld(\pi)=\gld(\pi)=1$. As a consequence of a more general theorem about ends of groups, proved by Stallings \cite{Sta} in the finitely generated case and extended to the general case by Swan \cite{Swa}, a group has $\cld(\pi)=1$ if and only if it is free (if and only if $\cat(\pi)=\gld(\pi)=1$). Hence we see that all three invariants are equal, except for the possibility of a group $\pi$ with $\cld(\pi)=\cat(\pi)=2$ and $\gld(\pi)=3$. The statement that such groups cannot exist has become known as the Eilenberg--Ganea conjecture, and remains unsolved.

Suppose now that a second discrete group $G$ acts on $\pi$ by automorphisms. Then $\pi$ will be called a \emph{group with operators in $G$}, or a \emph{$G$-group} for short. We will define the equivariant cohomological dimension, equivariant geometric dimension and equivariant LS category of such groups with operators, and show that the theorems of Eilenberg--Ganea and Stallings--Swan generalize to the equivariant setting when $G$ is finite.

Our definitions will employ the notions of Bredon cohomology and classifying spaces with respect to families of subgroups, which we now briefly recall (full definitions and references will be given in Section \ref{section: preliminaries} below). A non-empty collection $\F$ of subgroups of a group $\Gamma$ is called a \emph{family} if it is closed under conjugation and taking subgroups. A classifying space for $\Gamma$ with respect to the family $\F$ is a $\Gamma$-CW complex $E_\F(\Gamma)$ such that every $\Gamma$-space with isotropy in $\F$ admits a $\Gamma$-map $X\to E_\F(\Gamma)$, unique up to $\G$-homotopy. This is equivalent to asking that the fixed sub-complex $E_\F(\Gamma)^H$ is empty for $H\notin\F$, and contractible for $H\in \F$. Such a classifying space always exists, and is unique up to $\G$-homotopy type. The \emph{geometric dimension of $\G$ with respect to the family $\F$}, denoted $\gld_\F(\Gamma)$, is the minimal dimension of a classifying space $E_\F(\G)$.

The \emph{orbit category} $\OF\Gamma$ has as objects the $\Gamma$-sets $\Gamma/H$ for $H\in \F$ and as morphisms the $\Gamma$-maps. An $\OF\Gamma$-module is a contravariant functor from $\OF\Gamma$ to the category of abelian groups. The category of $\OF\G$-modules is an abelian category with enough projectives, so homological algebra machinery may be applied there. The \emph{cohomological dimension of $\G$ with respect to the family $\F$}, denoted $\cld_\F(\G)$, is defined to be the projective dimension of the constant $\OF\G$-module $\underline{\Z}$, which takes the value $\Z$ on all objects $\G/H$ and the identity on all morphisms. Equivalently, $\cld_\F(\G)$ is the least dimension $d$ such that $H^{d+1}(\OF\G;M) := \operatorname{Ext}^{d+1}_{\OF\G}(\ul{\Z},M)=0$ for all $\OF\G$-modules $M$.  Since the augmented cellular chain complex of a classifying space $E_\F(\G)$ is a projective resolution of $\underline{\Z}$, one has $\cld_\F(\G)\le\gld_\F(\G)$. By a theorem of L\"uck and Meintrup \cite{Lueck-Meintrup} one has $\gld_\F(\G)\le \sup\{3, \cld_\F(\G)\}$, generalizing part of the Eilenberg--Ganea result mentioned above.

Returning to the case of a $G$-group $\pi$, we consider the family of subgroups of the semi-direct product $\pi\rtimes G$ generated by the base group $1\times G\cong G$. That is, we let $\FG=\F\langle G\rangle$ denote the family of subgroups of $\pi\rtimes G$ which are conjugate to a subgroup of $G$.

\begin{defn}
The \emph{equivariant geometric dimension} and \emph{equivariant cohomological dimension} of the $G$-group $\pi$ are defined respectively by
\[
\gld_G(\pi):=\gld_\FG(\pi\rtimes G)\quad\mbox{and}\quad\cld_G(\pi):=\cld_\FG(\pi\rtimes G).
\]
\end{defn}

Recall that for a $G$-space $X$ with $G$ a compact Lie group, the equivariant LS category $\cat_G(X)$ was defined in \cite{Marz} to be the minimal integer $k$ for which there exists a cover of $X$ by open $G$-invariant subsets $U_0,\ldots , U_k$ such that each inclusion $U_i\hookrightarrow X$ is $G$-homotopic to a map with values in a single orbit. This notion is $G$-homotopy invariant. The equivariant LS category of the $G$-group $\pi$ is defined to be the equivariant LS category of a $G$-homotopy type of $K(\pi,1)$'s, described as follows. The classifying space $E_\FG(\GG)$ for the family $\FG$ described above is a contractible space on which $\pi\cong \pi\times 1$ acts freely, so the orbit space $E_\FG(\GG)/\pi$ is a $K(\pi,1)$. Since $E_\FG(\GG)$ is unique up to $(\GG)$-homotopy equivalence, the orbit space $E_\FG(\GG)/\pi$ is unique up to $G$-homotopy equivalence.

\begin{defn}
The \emph{equivariant LS category} of the $G$-group $\pi$ is defined to be
\[
\cat_G(\pi):=\cat_G(E_\FG(\GG)/\pi).
\]
\end{defn}

\begin{rem}
In Lemma \ref{Epimodel} below we show that as a model for the classifying space $E_\FG(\GG)$ we may take $E\pi$, the infinite join of copies of $\pi$, so that $\cat_G(\pi)=\cat_G(B\pi)$, the equivariant LS category of Milnor's classifying space for $\pi$. Note that $B\pi$, while infinite dimensional, may have the $G$-homotopy type of a finite complex. We remark that $E_\FG(\GG)/\pi$ (and in particular $B\pi$) is an Eilenberg--Mac Lane $G$-space of type $(\pi^{(-)},1)$ in the sense of Elmendorf \cite{Elmendorf}, where $\pi^{(-)}:\OG\to \mathsf{Grp}$ given by $G/H\mapsto \pi^H$ is the $\OG$-group determined by the system of fixed subgroups of $\pi$. Here $\OG$ is the orbit category for the family of all subgroups of $G$.
\end{rem}

With all these definitions in place, we can now state our main results.

\begin{thm}[Equivariant Eilenberg--Ganea Theorem]\label{EqEG}
Let $\pi$ be a discrete $G$-group, where $G$ is finite. Then the chain of inequalities
\[
\cld_G(\pi)\leq \cat_G(\pi)\leq \gld_G(\pi)\leq \sup\{3, \cld_G(\pi)\}
\]
is satisfied. Furthermore, if $\cld_G(\pi)=2$ then $\cat_G(\pi)=2$.
\end{thm}

In light of the general inequalities $\cld_\F(\Gamma)\leq \gld_\F(\G)\leq \sup\{3, \cld_\F(\G)\}$ alluded to above, the new contribution of this result is the definition of the equivariant LS category of a $G$-group, and its determination in terms of homological algebra. When the group $G$, or more generally its action on $\pi$, is trivial, we shall see below that the equivariant cohomological and geometric dimensions agree with their non-equivariant counterparts. Hence Theorem \ref{EqEG} generalizes the classical Eilenberg--Ganea theorem.

It is easily verified that $\cld_\F(\G)=0$ if and only if $\gld_\F(\G)=0$ if and only if $\F$ contains $\G$ (and therefore is the family of all subgroups of $\G$). Thus all three equivariant dimensions are zero precisely when the group $\pi$ is trivial. The second main result of this paper characterises $G$-groups of equivariant cohomological dimension one, assuming $G$ is finite.

\begin{thm}[Equivariant Stallings--Swan Theorem]\label{EqSS}
Let $\pi$ be a discrete $G$-group, where $G$ is finite. The following are equivalent:
\be
\item $\gld_G(\pi)=1$;
\item $\cat_G(\pi)=1$;
\item $\cld_G(\pi)=1$;
\item $\pi$ is a non-trivial free group with basis a $G$-set.
\ee
\end{thm}

When $G$ acts trivially, the implication $(3)\implies (4)$ is the Stallings--Swan theorem. Our proof of Theorem \ref{EqSS} uses a strengthening of Stallings--Swan due to Dunwoody \cite{Dunwoody}, which may be stated as follows. As is customary, we denote by $\ul\cld(\G)$ and $\ul\gld(\G)$ the cohomological and geometric dimensions of the group $\G$ with respect to the family $\FIN$ of finite subgroups. Then Dunwoody shows that $\ul\cld(\G)=1$ if and only if $\ul\gld(\G)=1$. When $G$ is finite and $\pi$ is torsion-free, the condition $\cld_G(\pi)=1$ entails that our family $\FG$ of subgroups of $\pi\rtimes G$ coincides with $\FIN$. This is not obvious, and is a consequence of the following theorem, which is our main algebraic result.

\begin{thm}\label{mainalg}
Let $\G$ be a finite group, and let $\F$ be any family of proper subgroups of $\G$. Then $\cld_\F(\G)\geq 2$.
\end{thm}

The following conjecture appears to be well known among experts (see \cite{Ono}).
\begin{conj}\label{conj:StallingsSwanFam}
If $\G$ is \emph{any} group and $\F$ is \emph{any} family of subgroups of $\G$, then $\cld_\F(\G)=1$ if and only if $\gld_\F(\G)=1$.
\end{conj}
To the best of our knowledge, this conjecture is proved in the literature only for the trivial family (by Stallings--Swan), the family $\FIN$ (by the result of Dunwoody mentioned above), and for the family $\VCYC$ of virtually cyclic subgroups, assuming $\G$ is countable (by a theorem of Degrijse \cite{Deg}). We were also informed that the forthcoming work of Petrosyan and Prytu\l a \cite{PP} shows that the conjecture holds for chamber transitive lattices in buildings with the family given by all the stabilisers. 

We observe that Theorem \ref{mainalg} verifies Conjecture \ref{conj:StallingsSwanFam} for a large class of families.
\begin{cor}
Let $\G$ be any group, and let $\F$ be any family of subgroups of $\G$ that does not contain the family $\FIN$ of finite subgroups. Then $\cld_\F(\G)=1$ if and only if $\gld_\F(\G)=1$.
\end{cor}
\begin{proof} Suppose $H\le\G$ is finite and not in $\F$. Then Shapiro's Lemma \ref{lemma: Shapiro} and Theorem \ref{mainalg} yield
\[
2\le \cld_{H\cap\F}(H)\le\cld_\F(\G).
\]
Thus $\cld_\F(\G)=1$ is impossible. On the other hand, $\gld_\F(\G)=1$ would imply that $\G$ acts on a tree without $H$-fixed points, which is also impossible \cite{Serre}.
\end{proof}

\begin{rem}\label{rem:equivgroupcohom}
Theorems \ref{EqEG} and \ref{EqSS} suggest that the equivariant group cohomology of the $G$-group $\pi$ with coefficients in an $\OFGG$-module $M$ should be defined by $H^*_G(\pi;M):=H^*(\OFGG;M)=\operatorname{Ext}^*_{\OFGG}(\underline{\Z},M)$. Given a $\GG$-module $N$, one obtains a $\OFGG$-module $N^{(-)}$ by taking fixed points (this is a form of co-induction). For such coefficient modules, our definition agrees with that of Inassaridze \cite{In}, who defines the equivariant group cohomology of $\pi$ with coefficients in $N$ to be $H^*(\GG,G;N)$, the relative group cohomology in the sense of Hochschild \cite{Hoch} and Adamson \cite{Adamson} (see also Benson \cite[Section 3.9]{Benson}). As observed in \cite[Section 2]{PY}, one has an isomorphism $H^*(\GG,G;N)\cong H^*(\OFGG; N^{(-)})$. Hence our definition generalizes that of Inassaridze by allowing as coefficients arbitrary $\OFGG$-modules which may not be co-induced from $\GG$-modules. In theory our $\cld_G(\pi)$ could exceed the equivariant cohomological dimension derived from Inassaridze's definition, but we do not currently know any examples where this is the case.

We mention also the paper of Cegarra--Garc\'ia-Calcines--Ortega \cite{CG-CO} which predates \cite{In} and contains a slightly different definition of equivariant group cohomology with coefficients in a $\GG$-module.
\end{rem}

Combining Theorems \ref{EqEG} and \ref{EqSS} gives the following Corollary and Question.

\begin{cor}
If $\pi$ is a discrete $G$-group with $G$ finite, then
\[
\cld_G(\pi)=\cat_G(\pi)=\gld_G(\pi),
\]
except for the possibility of a $G$-group $\pi$ with $\cld_G(\pi)=\cat_G(\pi)=2$ and $\gld_G(\pi)=3$. In particular, $\cat_G(\pi)=\cld_G(\pi)$ always.
\end{cor}

\begin{question}[Equivariant Eilenberg--Ganea Conjecture]
Does there exist a $G$-group $\pi$ with $\cld_G(\pi)=2$ and $\gld_G(\pi)=3$?
\end{question}

The structure of the paper is as follows. In Section 2 we recall necessary material on Bredon cohomology and cohomology of small categories in general, and derive some basic facts about equivariant dimensions as specializations. In Sections 3 and 4 we prove Theorems \ref{EqEG} and \ref{EqSS} respectively. The proof of Theorem \ref{EqSS} relies on Theorem \ref{mainalg}, whose proof is given in Section 5.

In writing the paper we have benefited from conversations with many people, including Dave Benson, Dieter Degrijse, Michael Farber, Ellen Henke, Ian Leary, Ran Levi, Assaf Libman, Greg Lupton, Brita Nucinkis, Bob Oliver and John Oprea. In particular, the first author thanks his co-authors of the paper \cite{FGLO} which inspired many of the results in Section 3 of the present paper.

\section{Preliminaries on Bredon cohomology}\label{section: preliminaries}

We now recall the necessary material on Bredon cohomology with respect to families, and cohomology of small categories more generally.
In this section $\G$ will denote an arbitrary discrete group.

\begin{defn} A set of subgroups $\F$ of $\G$ is called a \emph{family} if $\F$ is closed under conjugations and taking subgroups.
\end{defn}

It is often convenient to consider $\G$-CW complexes (see \cite[Section I.1]{Lueck}) with isotropy in the family $\F$. Such $\G$-CW complexes are $\G$-spaces built out of cells of type $\G/H \times D^n$, $n \in \Z$, where $H \in \F$. The \emph{classifying $\G$-space} $E_\F(\G)$ is the universal $\G$-CW complex with isotropy in $\F$ in the following sense: For any $\G$-CW complex $X$ with isotropy in $\F$, there is a continuous $\G$-map $X \to E_\F(\G)$ which is unique up to $\G$-homotopy. The $\G$-space $E_\F(\G)$ is uniquely characterised up to $\G$-homotopy equivalence by the following properties: $E_\F(\G)$ is a $\G$-CW complex and the $H$-fixed subspace $E_\F(\G)^H$ is contractible if $H \in \F$ and empty otherwise. Any such space is called \emph{a model for $E_\F(\G)$}.

There are many ways to construct such a classifying space, see for example \cite[Chapter I, Proposition 2.3]{Lueck} or \cite[Definition 2.1]{Lueck-Oliver}. To sketch the latter construction, we recall first the \emph{($\F$-)orbit category} $\OFG$. The category $\OFG$ has the cosets $\G/H$ with $H \in \F$ as objects and $\G$-equivariant maps as morphisms. One can consider a covariant functor from $\OFG$ to the category of $\G$-spaces sending $\G/H$ to $\G/H$ considered as a discrete $\G$-space. The Bousfield-Kan homotopy colimit of this functor is a model for $E_\F(\G)$.

The above described model for $E_\F(\G)$ is too big and is usually infinite dimensional. Often one can construct small models for $E_\F(\G)$. An especially well studied special case is $\F=\mathcal{FIN}$, the family of finite subgroups of $\G$. The $\G$-space $E_{\mathcal{FIN}}(\G)$ plays an important role in geometric group theory and algebraic and topological $K$-theory (via the Farell--Jones Conjecture and Baum--Connes Conjecture) and is often finite dimensional and cocompact. For example when $\G=\Z \rtimes \Z/2 $, the infinite dihedral group, a model for $E_{\mathcal{FIN}}(\G)$ is the real line $\mathbb{R}$ with the sign and translation actions. This is a one dimensional cocompact model for $E_{\mathcal{FIN}}(\G)$, whereas the construction in the previous paragraph yields an infinite dimensional space.

The latter example shows  that it makes sense to try to find a minimal model for $E_\F(\G)$. The first step towards this is to find the minimal dimension such a model can have. This is the geometric dimension of the group $\G$ with respect to the family $\F$, denoted by $\gld_\F(\G)$. More precisely,
\[\gld_\F(\G) :=\min \{ \dim X  \mid  X \mbox{ is a model for } E_{\F}(\G) \},\]
where $\dim$ stand for the CW-dimension.
To compute $\gld_\F(\G)$ one needs some homological algebra. With this goal in mind, we recall the definition of cohomology of a category with coefficients in a functor. For details we refer to \cite[Chapter II, Section 9 and Chapter III, Section 17]{Lueck}. We will mostly need this in the case of the orbit category, however we will also need cohomology of certain posets and other related categories.

Let $\CC$ be a small category and let $F : \CC^{op} \to \Ab$ be a functor into the category of abelian groups (i.e. a contravariant functor on $\CC$). Such a functor is referred to as a \emph{$\CC$-module}. The category of $\CC$-modules and natural transformations is denoted by $\CC-\Mod$. This category is an abelian category and has enough projective objects. Projective objects are direct summands of sums of representable modules (often referred to as free modules) which have the form
\[\bigoplus_{\alpha}\Z[\CC(-, C_{\alpha})], \]
where $\alpha$ runs over some indexing set and the $C_\alpha$ are the representing objects in $\CC$. Let $\ul{\Z} \colon \CC^{op} \to \Ab$ denote the constant module which assigns the value $\Z$ to every object in $\CC$ and the identity homomorphism to every morphism in $\CC$. The \emph{$n$-th cohomology of $\CC$ with coefficients in a $\CC$-module $F : \CC^{op} \to \Ab$} is defined to be the Ext-group
\[H^n(\CC; F):=\Ext^n_{\CC-\Mod}(\ul{\Z}, F).\]
(We will below shorten the notation $\Ext^n_{\CC-\Mod}$ to $\Ext^n_{\CC}$.) There is a more direct way without using homological algebra to define $H^n(\CC; F)$, using a certain bar construction. But this will not be needed in this paper and we do not recall the construction.

Next we recall the following well known definition.

\begin{defn} \label{definition: cd of a category} Let $\CC$ be a small category. The \emph{cohomological dimension} of $\CC$, denoted by $\cld(\CC)$, is the projective dimension of the constant module $\ul{\Z} \colon \CC^{op} \to \Ab$. Equivalently, $\cld(\CC)$ is equal to the minimum of lengths of projective resolutions of $\ul{\Z}$. Yet another equivalent definition uses Ext-groups:
\[\cld(\CC)= \max\{n \; \vert \; \Ext_{\CC}^n(\ul{\Z}, F) \neq 0 \; \text{for some} \; F \}.\]
\end{defn}

Now given a family of subgroups $\F$ of $\G$, we can specialise the above definitions to the orbit category $\OFG$. Given a functor $M \colon \OFG^{op} \to \Ab$ (also referred to as a \emph{coefficient system}), one gets the cohomology groups 
\[
H^*(\OFG; M):=\Ext_{\OFG}^*(\ul{\Z}, M).
\]
We are now ready to recall one of the most important definitions for this paper:

\begin{defn} \label{definition: cd of a family} The\emph{ cohomological dimension of $\G$ with respect to the family $\F$} is the (possibly infinite) number $\cld(\OFG)$ and is denoted by $\cld_\F(\G)$.
\end{defn}

There is a close connection between $\cld_\F(\G)$ and $\gld_\F(\G)$. By \cite[Theorem 0.1]{Lueck-Meintrup} one has the following inequalities:
\[ \cld_\F(\G)\leq \gld_\F(\G)\leq \sup\{3,\cld_\F(\G)\}. \]
The \emph{Eilenberg-Ganea conjecture} states that if $\cld(\G)=2$ (and hence is torsion-free), then $\G$ has $2$-dimensional $K(\G,1)$. It turns out that the analog of this conjecture for general families is false. Brady, Leary and Nucinkis showed in \cite{BLN} that for certain right-angled Coxeter groups $W(L)$ and the family $\F=\mathcal{FIN}$, the generalised Eilenberg-Ganea conjecture fails. In other words, they prove that $\cld_{\mathcal{FIN}}(W(L))=2$ but $\gld_{\mathcal{FIN}}(W(L))=3$.

If the family $\F$ contains the full subgroup $\G$, then it is easy to see that one may take as $E_\F(\G)$ a one-point space with the trivial $\G$ action. Consequently,
\[
\G\in\F \implies \gld_\F(\G)=0 \implies \cld_\F(\G)=0.
\]
Conversely, if $\cld_\F(\G)=0$ then \cite[Lemma 2.5]{Sym} of Symonds implies that $\F$ has a unique maximal element which is self-normalizing, and it follows that $\G\in\F$ (see \cite[Proposition 3.20]{Fluch}). Hence 
\[
\cld_\F(\G)=0\iff \gld_\F(\G)=0 \iff \Gamma\in\F. 
\]

It is conjectured that for a general family $\F$ one has $\cld_\F(\G)=1$ if and only if $\gld_\F(\G)=1$. For the trivial family this is known and it is the celebrated Stallings-Swan theorem \cite{Sta, Swa}. For the family $\mathcal{FIN}$ this conjecture also holds and it is the theorem of Dunwoody \cite{Dunwoody}. This paper addresses this conjecture for $\G$ a finite group by showing that for any proper family $\F$ one always has $\cld_\F(\G)>1$. We also prove the conjecture for the family $\FG$ of sub-conjugates of $G$ in the semi-direct product $\pi\rtimes G$, when $\pi$ is a $G$-group with $G$ finite.  

Next, we recall the definition of Bredon cohomology $H^*_{\F}(X; M)$ which generalises the cohomology groups $H^*(\OFG; M)$. Let $\G$ be a discrete group, $\F$ a family of subgroups, $X$ a $\G$-CW complex and $M \colon \OFG^{op} \to \Ab$ a coefficient system. The space $X$ gives a natural chain complex $\ul{C}_*(X)$ of $\OFG$-modules defined by $\ul{C}_*(X)(\G/H)=C_*(X^H)$, where $C_*(-)$ denotes the cellular chain complex with integer coefficients. The \emph{Bredon cohomology} of $X$ with coefficients in $M$ is defined by
\[
H^i_\F(X; M):=H^i\big(\Hom_{\OFG}(\ul{C}_*(X),M)\big).
\]
Here $\Hom_{\OFG}(\ul{C}_*(X),M)$ is the cochain complex of natural transformations  from $\ul{C}_*(X)$ to $M$. Given a model for $E_\F(\G)$, it follows from \cite[Lemma 2.6]{Lueck-Meintrup} that the chain complex $\ul{C}_*(E_\F(\G))$ of $\OFG$-modules is a free resolution of $\ul{\Z}$. This implies that there is a natural isomorphism:
\[H^i_\F(E_\F(\G);M) =  H^i\big(\Hom_{\OFG}(\ul{C}_*(E_\F(\G)),M)\big) \cong \Ext_{\OFG}^{i}(\ul{\Z}, M)= H^i(\OFG; M).\]

We now recall Shapiro's Lemma for families, which plays a fundamental r\^{o}le in this paper. Let $\G$ be a group, $\F$ a family of subgroups of $\G$ and $H$ a subgroup of $\G$. Then
\[H \cap \F=\{K \in \F \; \vert \; K \leq H \} \]
is a family of subgroups of $H$. Pre-composing with the functor
\[\G \times_H - : \mathcal{O}_{H \cap \F}H \to   \OFG\]
which sends $H/K$ to $\G/K$ induces the restriction functor
\[ \res_H^\G \colon \OFG-\text{Mod} \to \mathcal{O}_{H \cap \F}H-\text{Mod}.\]
Since it is induced by pre-composition, we get that $\res_H^\G$ is exact, preserves direct sums and sends $\ul{\Z}$ to $\ul{\Z}$. Moreover, a straightforward calculation shows that the following double coset formula of Mackey type holds:
\[ \res^\G_H(\Z[\OFG(-, \G/K)] ) \cong \bigoplus_{g \in H \setminus \G/K} \Z[\mathcal{O}_{H \cap \F}H(-, H/H \cap {}^gK)].\]
This implies that $\res_H^\G$ preserves projective resolutions of $\ul{\Z}$.
The functor $\res_H^\G$ has a right adjoint
\[
\operatorname{coind}_H^\G:\mathcal{O}_{H \cap \F}H-\text{Mod}\to \OFG-\text{Mod}
\]
called co-induction (see \cite[Chapter 1, Section 10]{Fluch}, for example). One has the following generalization of the well-known Shapiro's Lemma.

\begin{lem} \label{lemma: Shapiro} Let $\F$ be a family of subgroups of a group $\G$, and let $H$ be a subgroup of $\G$. Then for all $M\in\mathcal{O}_{H \cap \F}H-\text{Mod}$ and $n\in \Z$ one has isomorphisms
\[
H^n(\mathcal{O}_{H \cap \F}H;M) \cong H^n(\OFG; \operatorname{coind}_H^\G(M)),
\]
which are natural in $M$. Consequently,
\[\cld_{H \cap \F}(H) \leq \cld_\F(\G).\]
\end{lem}

Shapiro's Lemma has a geometric counterpart, which is trivial to prove but nevertheless useful.

\begin{lem} \label{lemma: geometric Shapiro}
Let $\F$ be a family of subgroups of a group $\G$, and let $H$ be a subgroup of $\G$. Then any model for the classifying space $E_\F(\G)$ is also a model for $E_{H\cap\F}(H)$. Consequently,
\[
\gld_{H\cap\F}(H)\leq\gld_\F(\G).
\]
\end{lem}

\begin{proof}
Let $X$ be a model for $E_\F(\G)$. By restriction of the action, $X$ becomes an $H$-CW complex. Given $K\in H\cap\F$, since $K\in \F$ we have that $X^K$ is (weakly) contractible. Given $K\leq H$ with $K\notin H\cap\F$, we must have $K\notin \F$ and therefore $X^K$ is empty.
\end{proof}

Recall that in the Introduction we have defined the equivariant cohomological and geometric dimensions
\[
\cld_G(\pi) :=\cld_\FG(\pi\rtimes G)\qquad\mbox{and}\qquad\gld_G(\pi):=\gld_\FG(\pi\rtimes G),
\]
where $\pi$ is a discrete $G$-group and $\FG$ is the family of sub-conjugates of $G$ in the semi-direct product $\pi\rtimes G$. We observe that $\pi\rtimes G\in \FG$ if and only if $\pi$ is trivial, and so 
\[
\cld_G(\pi)=0 \iff \gld_G(\pi)=0\iff \pi\mbox{ is trivial.}
\] 
Since $\pi\cap\FG=\{1\}$, Shapiro's Lemma gives
\[
\cld(\pi)\le\cld_G(\pi)\qquad\mbox{and}\qquad \gld(\pi)\leq\gld_G(\pi).
\]
Thus the equivariant dimensions are bounded below by the non-equivariant dimensions (and are infinite if $\pi$ contains torsion elements). 
When $G$ acts trivially on $\pi$, both inequalities become equalities. For in this case, $\FG$ is the family of subgroups of the normal subgroup $G\nsub \pi\rtimes G\cong\pi\times G$, and we have the following general result.
\begin{lem}\label{lem:trivialaction}
Let $N\nsub\G$ be a normal subgroup, and let $\mathcal{F}=\mathcal{F}\langle N\rangle$ be the family of subgroups of $\G$ which are contained in $N$. Then
\[
\cld_\mathcal{F}(\G)=\cld(\G/N)\qquad\mbox{and}\qquad \gld_\mathcal{F}(\G)=\gld(\G/N).
\]
\end{lem}
\begin{proof}
We have an inclusion functor $F:\G/N\to \OFG$, where $\G/N$ is regarded as a category with one object. Associated to $F$ are two functors 
\[
\res_F: \OFG-\text{Mod}\to \G/N-\text{Mod},\qquad M\mapsto M(\G/N)
\] 
and
\[
\operatorname{ind}_F: \G/N-\text{Mod}\to \OFG-\text{Mod},\qquad P\mapsto\left(\G/H\mapsto P^H=P\right),
\]
where in the second definition we first regard $P$ as a $\Gamma$-module via the projection $\G\to \G/N$ and then take fixed points. The reader can verify that both $\res_F$ and $\operatorname{ind}_F$ are exact, preserve direct sums and free modules, and send constant $\ul{\Z}$ to constant $\ul{\Z}$. This gives the first equality. 

It is easily verified that any  model for $E(\G/N)$, regarded as a $\G$-CW complex via the quotient map $\G\to \G/N$, is model for $E_\F(\G)$. Conversely, if $X$ is a model for $E_\F(\G)$ then $X^N$ is a contractible complex on which $\G/N$ acts freely, hence a model for $E(\G/N)$. This gives the second equality.
\end{proof}

\begin{rem}
An alternative definition of the equivariant cohomological dimension of a $G$-group can be given, using Inassaridze's definition of equivariant group cohomology \cite{In} as recalled in Remark \ref{rem:equivgroupcohom}. Given a $(\pi\rtimes G)$-module $N$, we obtain an $\mathcal{O}_\mathcal{G}(\pi\rtimes G)$-module $N^{(-)}$ by taking fixed sub-modules over each orbit. Define
\[
\cld_G^I(\pi)=\sup\{n\mid H^n(\mathcal{O}_\mathcal{G}(\pi\rtimes G);N^{(-)})\neq 0\mbox{
for some }(\pi\rtimes G)\mbox{-module }N\}.
\]
Co-induction along the functor $\mathcal{O}_{\{1\}}\pi\to \mathcal{O}_\mathcal{G}(\pi\rtimes G)$ sends a $\pi$-module $M$ to $\left(\operatorname{Hom}_\pi(\Z[\pi\rtimes G],M)\right)^{(-)}$. In other words, it is the composition of the usual co-induction from $\pi$-modules to $(\pi\rtimes G)$-modules with taking fixed sub-modules. An argument with Shapiro's Lemma therefore yields the first inequality below:
\[
\cld(\pi)\le\cld_G^I(\pi)\le\cld_G(\pi).
\]
We do not know whether the second inequality, which is immediate from the definitions, can be strict.
\end{rem}

\section{The Equivariant Eilenberg--Ganea Theorem}

In this section we give the proof of Theorem \ref{EqEG}. Recall that $\pi$ is a discrete $G$-group, where $G$ is a finite group. We denote the image of an element $\alpha\in \pi$ under $g\in G$ by ${}^g \alpha$. The semi-direct product $\pi\rtimes G$ has group multiplication given by $(\alpha,g)\cdot(\beta,h)=(\alpha{}^g\beta,gh)$.

As a discrete space, $\pi$ admits left actions of $\pi$ (induced by the group operation) and $G$ (given by the action). These actions are compatible, in the sense that for all $g\in G$ and $\alpha,\beta\in \pi$ we have ${}^g(\alpha \beta) = {}^g \alpha  {}^g \beta$, and so we get a left action of the semi-direct product $\pi\rtimes G$ on $\pi$, given by
\[
(\alpha,g)\cdot \alpha_0 = \alpha{}^g \alpha_0,\qquad g\in G, \quad \alpha,\alpha_0\in \pi.
\]

For $k\ge0$ let $E_k\pi$ denote the $(k+1)$-fold topological join of the discrete space $\pi$. Note that $E_k\pi$ is naturally a $k$-dimensional simplicial complex of the homotopy type of a wedge of $k$-spheres. The $(\GG)$-action on $\pi$ extends diagonally to an action on $E_k\pi$, making it into a $(\GG)$-CW complex. Taking the colimit of the obvious inclusions $E_k\pi\hookrightarrow E_{k+1}\pi = (E_k\pi)\ast\pi$, we obtain the infinite join $E\pi = \bigcup_{k\ge0} E_k\pi$ as an infinite dimensional $(\GG)$-CW complex.

\begin{lem}\label{Epimodel}
The space $E\pi$ is a model for $E_\FG(\pi\rtimes G)$.
\end{lem}

\begin{proof}
We must show that the isotropy of $E\pi$ lies in $\FG$, and that for each $H\in \FG$ the fixed point set $(E\pi)^H$ is contractible.

We use a standard notation in which elements of the infinite join $E\pi$ are represented as (non-commutative) formal sums $\sum t_i\alpha_i$ with $t_i\in [0,1]$ almost all zero, $\sum t_i=1$ and $\alpha_i\in\pi$ for all $i$. Then the action is given by $(\alpha,g)\cdot \sum t_i\alpha_i = \sum t_i \alpha {}^g \alpha_i$. 

Let $H\leq \GG$ denote the stabiliser of $\sum t_i\alpha_i\in E\pi$. Choose an index $i$ such that $t_i>0$, and note that for all $(\alpha,g)\in H$ we have $\alpha{}^g \alpha_i=\alpha_i$. One verifies that
\[
(\alpha_i^{-1},1)(\alpha,g)(\alpha_i,1) = (\alpha_i^{-1}\alpha{}^g \alpha_i,g)= (1,g),
\]
so that $H$ is conjugate in $\GG$ to a subgroup of $G$. It follows that $E\pi$ has all isotropy groups in the family $\FG$.

Now suppose $H\leq G$. There is an evident homeomorphism $(E\pi)^H\cong E(\pi^H)$, hence $(E\pi)^H$ is contractible. Its translates $(\alpha,g)E(\pi^H) = (E\pi)^{(\alpha,g)H(\alpha,g)^{-1}}$ are therefore also contractible. Hence the fixed-point sets are contractible for all groups in $\mathcal{G}$, and $E\pi$ is a model for $E_\mathcal{G}(\pi\rtimes\G)$ as claimed.
\end{proof}

Let $B\pi=(E\pi)/\pi$, the orbit space of the free (left) $\pi$-action on $E\pi$. The $G$-action on $E\pi$ descends to a $G$-action on the quotient $B\pi$, and we have defined $\cat_G(\pi):=\cat_G(B\pi)$.

\begin{defn}[{\cite{CG,G}}]
Given a $G$-fibration $p:E\to B$, the equivariant sectional category, denoted $\secatG(p)$, is the minimal integer $k$ for which there exists a cover of $B$ by $G$-invariant open sets $U_0,\ldots , U_k$, on each of which $p$ admits a local $G$-section (i.e., a continuous $G$-map $s_i:U_i\to E$ such that $p\circ s_i=\operatorname{incl}:U_i\hookrightarrow B$).
\end{defn}

For the definition of $G$-fibration, see \cite[p.53]{tD}. Let $p:E\pi\to B\pi$ be the quotient map. Then $p$ is a $G$-fibration (since it is a locally trivial $(\pi,\alpha,G)$-bundle over a $G$-paracompact base; compare \cite[Chapter I, Exercise 7.5.5]{tD}).

\begin{prop}\label{prop:cat=secat}
The equivariant category $\catG(B\pi)$ is equal to $\secat_G(p)$, where $p:E\pi\to B\pi$ is the quotient map.
\end{prop}

\begin{proof}
It is shown in \cite[Corollary 4.7]{CG} that if $q:E\to B$ is a $G$-fibration such that:
\begin{enumerate}[(i)]
\item $E$ is $G$-categorical (i.e., the identity map on $E$ is $G$-homotopic to a map with values in a single orbit); and
\item $q(E^H)=B^H$ for all subgroups $H\le G$;
\end{enumerate}
then $\secatG(q)=\catG(B)$. We will show that conditions (i) and (ii) hold for $q=p:E\pi\to B\pi$.

We have shown in Lemma \ref{Epimodel} that $E\pi$ is a model for $E_\mathcal{G}(\pi\rtimes G)$. It follows that $E\pi$ is also a model for $E_{G\cap \FG}(G)$ (see Lemma \ref{lemma: geometric Shapiro}). However, $G\cap \FG=\mathcal{ALL}$ is the family of \emph{all} subgroups of $G$, and so $E\pi$ is $G$-homotopy equivalent to a point, and in particular is $G$-categorical. Hence (i) is satisfied.

Next, let $H\le G$ be any subgroup. Clearly $p((E\pi)^H)\subseteq (B\pi)^H$, and we must show surjectivity. So let $x\in (B\pi)^H$. Since $p$ is surjective, there exists $y\in p^{-1}(x)\subseteq E\pi$. Although $y$ need not be fixed by $H$, for all $g\in H$ there exists a unique (since $\pi$ acts freely on $E\pi$) element $\alpha_g\in \pi$ such that $\alpha_g {}^g y = y$. Representing $y$ as a formal sum $\sum t_i \alpha_i$, we find that for every $i$ such that $t_i>0$, and for all $g\in H$, the equation
\begin{equation}\label{ggamma}
\alpha_g {}^g \alpha_i = \alpha_i\qquad\mbox{equivalently,}\qquad {}^g \alpha_i = \alpha_g^{-1} \alpha_i,
\end{equation}
holds. Let $j$ be any specific index such that $t_{j}>0$. We claim that $\alpha_{j}^{-1}y\in (E\pi)^H$; as $p(\alpha_{j}^{-1}y)=x$, this verifies condition (ii). This is a straightforward calculation using Equation (\ref{ggamma}): for all $g\in H$,
\begin{align*}
{}^g \left(\sum t_i \alpha_{j}^{-1}\alpha_i\right) & = \sum t_i {}^g\left(\alpha_{j}^{-1} \alpha_i\right) \\
                                                & = \sum t_i {}^g\left( \alpha_{j}^{-1}\right) {}^g \left(\alpha_i\right) \\
                                                & = \sum t_i \left(\alpha_j^{-1} \alpha_g\right)\left( \alpha_g^{-1} \alpha_i\right)\\
                                                & = \sum t_i \alpha_j^{-1} \alpha_i.
\end{align*}
\end{proof}

\begin{cor}\label{cor:equivmap}
The equivariant category $\catG(B\pi)$ equals the minimal integer $k$ such that there exists a $(\GG)$-equivariant map $E\pi\to E_k\pi$.
\end{cor}

\begin{proof}
By Proposition \ref{prop:cat=secat} we have $\catG(B\pi)=\secatG(p)$, where $p:E\pi\to B\pi$ is the quotient map. We use the characterization of $G$-sectional category in terms of $G$-sections of fibred joins, observed in \cite[Proposition 3.4]{G}. In particular, since $p : E\pi \to B\pi$ is a $G$-fibration over a paracompact base space, $\secatG(p) \leq k$ if and only if the $(k+1)$-fold fibred join $p_k : J^k_{B\pi}(E\pi)\to B\pi$ admits a (global) $G$-section.

The $G$-fibration $p:E\pi\to B\pi$ can be identified with the associated fibration $q:E\pi\times_\pi \pi\to B\pi$ with fibre $\pi$, as follows. Sticking with left actions, the total space $E\pi\times_\pi \pi$ is  the orbit space of $E\pi\times \pi$ under the diagonal $\pi$-action given by $$\left(\alpha, \left(\sum t_i \alpha_i, \beta\right)\right) \mapsto \left(\sum t_i\alpha\alpha_i , \alpha \beta\right).$$ There is a $G$-homeomorphism $\phi: E\pi\times_\pi \pi\to E\pi$ given by $\left[\sum t_i \alpha_i,\beta\right]\mapsto \sum t_i \beta^{-1}\alpha_i$, where the action of $G$ on $E\pi\times_\pi \pi$ is given by $(g, \left[\sum t_i \alpha_i,\beta\right])\mapsto \left[\sum t_i {}^g \alpha_i ,{}^g \beta\right]$. This action preserves the fibres of the projections to $B\pi$.

It follows that $\secatG(p)\leq k$ if and only if the $(k+1)$-fold fibred join $q_k$ of $q: E\pi\times_\pi \pi\to B\pi$ admits a $G$-section. By Schwarz \cite[Proposition 1]{Schwarz}, $q_k$ can be identified with the associated fibration $Q_k: E\pi\times_\pi E_k\pi\to B\pi$ with fibre the $(k+1)$-fold join $E_k\pi$. Here the $G$-action on the total space is given by
\[
\left(g, \left[\sum t_i \alpha_i, s_0\beta_0 + \cdots + s_k \beta_k\right]\right)\mapsto \left[\sum t_i {}^g \alpha_i, s_0 {}^g \beta_0 + \cdots + s_k {}^g \beta_k\right].
\]
Sections of $Q_k$ correspond to $\pi$-maps $E\pi\to E_k\pi$, while $G$-sections of $Q_k$ correspond to $(\GG)$-maps $E\pi\to E_k\pi$. More explicitly, given a $(\GG)$-map $\psi: E\pi\to E_k\pi$, we obtain a $G$-section $\sigma:B\pi\to E\pi\times_\pi E_k\pi$ of $Q_k$ by setting $\sigma[e]=[e,\psi(e)]$ for $e\in E\pi$. Conversely, given a $G$-section $\sigma:B\pi\to E\pi\times_\pi E_k\pi$ we define $\psi:E\pi\to E_k\pi$ using the formula $\sigma[e]=[e,\psi(e)]$. Checking that $\psi$ is a $(\GG)$-map is straightforward.
\end{proof}

\begin{cor}\label{cor:catGhtpyretract}
The equivariant category $\catG(B\pi)$ equals the minimal integer $k$ such that $E\pi$ is a $(\GG)$-homotopy retract of a $(\GG)$-CW complex of dimension $k$.
\end{cor}

\begin{proof}
Suppose $\catG(B\pi)\leq k$. By Corollary \ref{cor:equivmap}, there exists a $(\GG)$-map $\psi:E\pi\to E_k\pi$. By Lemma \ref{Epimodel}, the space $E\pi$ is a classifying space $E_\FG(\GG)$ for the family $\FG$. Since $E_k\pi$ is a sub-complex of $E\pi$, it too has isotropy in $\FG$, and therefore there is a classifying $(\GG)$-map $\phi:E_k\pi\to E\pi$. Since $(\GG)$-maps $E\pi\to E\pi$ are unique up to $(\GG)$-homotopy, the composition $\phi\circ \psi$ is $(\GG)$-homotopic to the identity. Therefore $E\pi$ is a $(\GG)$-homotopy retract of $E_k\pi$, which has dimension $k$.

Conversely, suppose we have a factorisation
\[
\xymatrix{
E\pi \ar[r]^\psi & L \ar[r]^\phi & E\pi
}
\]
of the identity map up to $(\GG)$-homotopy, where $L$ is a $(\GG)$-CW complex of dimension $k$. Observe that this implies that $L^H=\varnothing$ for subgroups $H\le\GG$ not in $\FG$. Let $f:E_k\pi\to E\pi$ denote the inclusion. We use the equivariant Whitehead Theorem (see \cite[Theorem II.2.6]{tD} or \cite[Theorem I.3.2]{May}, for example) to show that the map $L\to E\pi$ factors through $f$ up to $(\GG)$-homotopy. For this let $\nu:\operatorname{Con}(\GG)\to \Z$ be the function on conjugacy classes of subgroups of $\GG$ given by
\[
\nu(H)=\begin{cases} k & \mbox{if }H\in \FG, \\ -1 & \mbox{if }H\notin\FG, \end{cases}
\]
and observe that $L$ has dimension at most $\nu$ and that $f$ is a $\nu$-equivalence. Therefore
\[
f_\ast:[L,E_k\pi]_{\GG}\to [L,E\pi]_{\GG}
\]
is surjective. We therefore have a $(\GG)$-map $E\pi\to L\to E_k\pi$, and by Corollary \ref{cor:equivmap} this implies that $\catG(B\pi)\leq k$.
\end{proof}

We are now in a position to prove Theorem \ref{EqEG}, restated here for convenience.

\begin{thm}[Equivariant Eilenberg--Ganea Theorem]\label{EqEGrestatement}
Let $\pi$ be a discrete $G$-group, where $G$ is finite. Then the chain of inequalities
\[
\cld_G(\pi)\leq \cat_G(\pi)\leq \gld_G(\pi)\leq \sup\{3, \cld_G(\pi)\}
\]
is satisfied. Furthermore, if $\cld_G(\pi)=2$ then $\cat_G(\pi)=2$.
\end{thm}

\begin{proof}
As noted above, the inequalities $\cld_G(\pi)\le \gld_G(\pi)\leq \sup\{3, \cld_G(\pi)\}$ follow from the more general \cite[Theorem 0.1]{Lueck-Meintrup} applied to the family $\FG$.

Suppose $\gld_G(\pi)\le k$, meaning there is a $k$-dimensional $(\GG)$-CW complex $L$ which is a model for $E_\FG(\GG)$. By uniqueness of classifying spaces and Lemma \ref{Epimodel}, there is a $(\GG)$-homotopy equivalence $E\pi \simeq L$. In particular $E\pi$ is a $(\GG)$-homotopy retract of $L$, and $\catG(\pi)=\catG(B\pi) \leq k$ by Corollary \ref{cor:catGhtpyretract}. Hence $\cat_G(\pi)\leq\gld_G(\pi)$.

Now suppose that $\catG(\pi)\leq k$. By Corollary \ref{cor:catGhtpyretract} the identity map on $E\pi$ factors up to $(\GG)$-homotopy through a $(\GG)$-CW complex $L$ of dimension $k$. Then for any $i>k$ and $\OFGG$-module $M$, the identity homomorphism on the Bredon cohomology group $H^i_\FG(E\pi;M)\cong H^i(\OFGG;M)$ factors through $H^i_\FG(L;M)=0$. Hence $\cld_G(\pi)\leq k$.

To prove the final statement, we invoke equivariant obstruction theory. First note that in order to prove that $\cat_G(\pi)\le 2$, it is sufficient to show the existence of a $(\GG)$-equivariant map $E\pi\to L$, where $L:=E\pi^{(2)}$ denotes the $2$-skeleton of $E\pi$. For given such a map, composing with the inclusion $L\hookrightarrow E\pi$ gives a map $E\pi\to E\pi$, which by uniqueness of classifying maps must be $(\GG)$-homotopic to the identity. Hence $E\pi$ is a $(\GG)$-homotopy retract of the $2$-dimensional complex $L$, and we invoke Corollary \ref{cor:catGhtpyretract}.

Note that for all $H\in\FG$, the fixed subcomplex $L^H$ equals the $2$-skeleton of the contractible space $(E\pi)^H$, hence is simply-connected, and in particular $n$-simple for all $n$. The obstructions to the existence of an equivariant map $E\pi\to L$ lie in Bredon cohomology groups
$$
H^{n+1}_\FG(E\pi;\ul{\pi}_n(L))\cong H^{n+1}(\OFGG;\ul{\pi}_n(L)),
$$
where $\ul{\pi}_n(L):\OFGG\to \Ab$ is defined by $\ul{\pi}_n(L)(\GG/H)=\pi_n(L^H)$ for all $H\in\FG$
(compare \cite[Section 5]{May}, \cite[Theorem 2.6]{Lueck-Oliver}). When $\cld_G(\pi)=2$, these groups are trivial for $n\ge2$, and they are trivial for $n\le 1$ by the simple-connectivity of $L^H$ alluded to above.
\end{proof}

\section{The Equivariant Stallings--Swan Theorem}

In this section we will prove Theorem \ref{EqSS} from the introduction, restated below as Theorem \ref{EqSSrestatement}. The proof relies on Theorem \ref{mainalg} (proved in the next section), as well as the concept of non-abelian cohomology to relate the family $\FG$ of subgroups of $\GG$ to the family $\FIN$ of finite subgroups. We use the standard notations $\ul{\cld}(\G):=\cld_\FIN(\G)$ and $\ul{\gld}(\G):=\gld_\FIN(\G)$. 

Recall that a \emph{$1$-cocycle} $\varphi:G\to \pi$ is a function satisfying $\varphi(gh)=\varphi(g){}^g \varphi(h)$ for all $g,h\in G$. Define an equivalence relation on $1$-cocycles by declaring $\varphi_1\sim \varphi_2$ if there exists $\alpha\in \pi$ such that $\varphi_1(g) =\alpha^{-1} \varphi_2(g) {}^g \alpha$ for all $g\in G$. The set of equivalence classes is denoted by $H^1(G;\pi)$, and called the \emph{first non-abelian cohomology of $G$ with coefficients in $\pi$}. A $1$-cocycle is called \emph{principal} if it has the form $\varphi=\varphi_\alpha$ for some $\alpha\in \pi$, where $\varphi_\alpha(g) = \alpha{}^g(\alpha^{-1})$ for all $g\in G$. Note that principal $1$-cocycles are all equivalent to the trivial $1$-cocycle which is constant at the identity of $\pi$. Thus $H^1(G;\pi)$ is naturally based by the class of principal $1$-cocycles, which we denote by $1$. 

\begin{prop}\label{nonabelian}
 Assume that $G$ is finite and that $\pi$ is torsion-free. Then $\mathcal{G}=\mathcal{FIN}$ if and only if $H^1(H;\pi)=\{1\}$ for all subgroups $H\leq G$.
\end{prop}

\begin{proof}
Assume $\mathcal{G}=\mathcal{FIN}$. Let $\varphi: H\to \pi$ be a $1$-cocycle. We obtain a finite subgroup $H_\varphi$ of $\pi\rtimes G$ by setting $H_\varphi=\{(\varphi(h),h)\mid h\in H\}$, which by assumption is conjugate to a subgroup of $G$. Thus there exists some $(\alpha,g)\in \GG$ for which $(\alpha,g)^{-1} H_\varphi (\alpha,g)\leq 1\times G$. This means that for all $h\in H$, we have
\[
(\alpha,g)^{-1} (\varphi(h),h) (\alpha,g)  = ({}^{g^{-1}}(\alpha^{-1}\varphi(h){}^h \alpha),g^{-1}h g) \in 1\times G,
\]
and therefore $\alpha^{-1}\varphi(h){}^h \alpha=1\in \pi$, or $\varphi(h) = \alpha {}^h(\alpha^{-1})$, and $\varphi$ is principal. Hence $H^1(H;\pi)=\{1\}$ as claimed.

Conversely, assume $H^1(H;\pi)=\{0\}$ for all $H\leq G$. Since $G$ is finite, $\mathcal{G}\subseteq\mathcal{FIN}$. So let $\widetilde{H}\leq \pi\rtimes G$ be finite; we must show that $\widetilde{H}\in \mathcal{G}$. Let $p:\pi\rtimes G\to G$ denote the projection, and let $H:=p(\widetilde{H})\leq G$. Observe that, since $\pi$ is torsion-free, the intersection $\pi\cap \widetilde{H}$ is trivial. By the characterisation of subgroups of semi-direct products (described for instance by Usenko \cite{Usenko}), there exists a $1$-cocycle $\varphi_H:H\to \pi$ such that $\widetilde{H}=\{(\varphi_H(h),h)\mid h\in H\}$. By assumption, $\varphi_H$ is principal. Thus there exists $\alpha\in \pi$ such that
\begin{align*}
\widetilde{H}  & = \{(\alpha{}^h(\alpha^{-1}) ,h)\mid h\in H\}\\
   & = (\alpha,1)(1,H) (\alpha,1)^{-1}.
\end{align*}
Hence $\widetilde{H}\in \mathcal{G}$ as claimed.
\end{proof}

\begin{ex}\label{ex:dihedral}
Let $\pi=\Z$ with $G=\Z/2$ acting by the sign automorphism. Then $\pi\rtimes G=\Z\rtimes \Z/2$ is the infinite dihedral group, and since $H^1(G;\pi)$ is of order $2$ we have $\FG\neq\FIN$ in this case. 

In fact, there exist subgroups $H\leq \pi\rtimes G$ isomorphic to $\Z/2$ which are not conjugate to $G$. By Shapiro's Lemma \ref{lemma: Shapiro}, we have 
\[
\infty=\cld(\Z/2)=\cld_{H\cap\FG}(H)\leq\cld_\FG(\pi\rtimes G)=\cld_G(\pi).
\]
Note that $\ul{\cld}(\pi\rtimes G)=\ul{\gld}(\pi\rtimes G)=1$ in this case.

More generally, if there exists a finite subgroup $H\le\pi\rtimes G$ such that $K\notin\FG$ for every non-trivial subgroup $1 \neq K\le H$, then $\cld_G(\pi)=\gld_G(\pi)=\infty$.
\end{ex}

\begin{ex}
The following example, contained in \cite{Leary-Nucinkis}, illustrates that $\cld_G(\pi)$ can be finite even when $\FG\neq\mathcal{FIN}$. Furthermore, in this example one has
\[
\ul{\cld}(\pi\rtimes G)=3>2=\cld_G(\pi),
\]
which together with Example \ref{ex:dihedral} illustrates that we cannot expect a general inequality between the proper cohomological dimension of the semi-direct product and the equivarant cohomological dimension. We thank an anonymous referee for pointing out the second equality above and its proof.

There is an admissible action of $G=A_5$ on an acyclic $2$-dimensional flag complex $L$ without $G$-fixed points, such that $L^H$ is acyclic (and in particular non-empty) for all proper subgroups $H<A_5$. This fundamental example due to Floyd and Richardson \cite{FR} appears in many papers, including \cite{Adem,BDP,BB,BLN}.

 Let $R_L$ be the right-angled Artin group associated to the $1$-skeleton $L^{(1)}$, and let $\pi=H_L$ be the associated Bestvina--Brady group, that is, the kernel of the map $\phi: R_L\to \Z$ which maps all generators to $1\in \Z$. The action of $A_5$ on $L$ induces action of $A_5$ on $R_L$ and on $H_L$. For the semi-direct product $H_L\rtimes A_5$, it is shown in \cite[Theorem 3]{Leary-Nucinkis} that:

\be[(a)]
\item there are infinitely many conjugacy classes of subgroups $\overline{A_5}\leq H_L\rtimes A_5$ which project isomorphically to $A_5$;
\item all subgroups $\overline{H}\leq H_L\rtimes A_5$ which project isomorphically to a conjugate of proper subgroup $H<A_5$ are conjugate.
\ee
Hence $\mathcal{G}\neq\mathcal{FIN}$, by item (a). We remark that \cite[Theorem 6]{Leary-Nucinkis} gives 
\[
{\sf vcd}(H_L\rtimes A_5)=\cld(H_L)=2,\qquad \ul{\cld}(H_L\rtimes A_5)=\ul{\gld}(H_L\rtimes A_5)=3.
\]

We will show that $\cld_{A_5}(H_L)=2$. For this we recall some facts about the universal cover of the Salvetti complex of the right-angled Artin group $R_L$, which is a ${\rm CAT}(0)$ (hence contractible) cube complex $X$ on which $R_L$ acts freely \cite{BB,Leary-Nucinkis}. Choosing a base point $x_0\in X$ allows us to index the $k$-cubes of $X$ as $(\alpha x_0,\{v_1,\ldots, v_k\})$, where $\alpha \in R_L$ and the vertices $v_1,\ldots , v_k$ form a $k$-clique in $L^{(1)}$. In particular, the vertices of $X$ are indexed as $\alpha x_0$ with $\alpha\in R_L$. The $R_L$-action is given by $\beta\cdot(\alpha x_0,\{v_1,\ldots, v_k\})=(\beta\alpha x_0,\{v_1,\ldots , v_k\})$. There is an $R_L$-equivariant map $h:X\to \R$, where $\R$ is given the standard structure of a $\Z$-CW complex with $0$-cells $n\in\Z$, and $R_L$ acts via the homomorphism $\phi:R_L\to \Z$ and translation. In fact one can define $h(\alpha x_0)= \phi(\alpha)$ on vertices, then extend in an affine way to each cube. The level set $Q:= h^{-1}(0)\subseteq X$ is a sub-complex which is no longer cubical, but carries the structure of an $H_L$-CW complex. In the example at hand, $X$ is $3$-dimensional and $Q$ is $2$-dimensional. 

The group $A_5$ acts on $R_L$ and $H_L$ as described above, and also on $X$ via $g\cdot(\alpha x_0,\{v_1,\ldots , v_k\})=({}^g\alpha x_0,\{{}^g v_1,\ldots , {}^gv_k\})$. This action is compatible with the action of $R_L$, and hence $X$ becomes a $(R_L\rtimes A_5)$-CW complex. The sub-complex $Q$ is in the same way an $(H_L\rtimes A_5)$-CW complex.

The cellular chain complex $\ul{C}_*(Q)$ of $Q$ is a free resolution of $\ul{\Z}$ of length $2$ in the category of $\mathcal{O}_\mathcal{G}(H_L\rtimes A_5)$-modules. To prove this it suffices to show that for any subgroup $\G\leq H_L\rtimes A_5$, the fixed point set $Q^\G$ is acyclic if $\G\in\mathcal{G}$ and empty if $\G\notin \mathcal{G}$. 

The $2$-complex $Q$ is acyclic, by \cite[Corollary 7.2]{BB}. Thus by a theorem of Segev \cite{Segev,Adem}, each $Q^\G$ is either acyclic or empty. If $\G\in\mathcal{G}$ then $\G=\alpha H\alpha^{-1}$ for some $\alpha \in H_L$ and $H\leq A_5$. Then $Q^\G$ contains the vertex $\alpha x_0$, and is therefore acyclic. Conversely, if $Q^\G$ is non-empty it contains some vertex $\alpha x_0$ with $\alpha\in H_L$. Then $\G\leq \operatorname{Stab}(\alpha x_0)=\alpha A_5 \alpha^{-1}$, and it follows that $\G\in\mathcal{G}$.  \qed

\end{ex}

We now turn to the proof of Theorem \ref{EqSS}, restated here for convenience.

\begin{thm}[Equivariant Stallings--Swan Theorem]\label{EqSSrestatement}
Let $\pi$ be a discrete $G$-group, where $G$ is finite. The following are equivalent:
\be
\item $\gld_G(\pi)=1$;
\item $\cat_G(\pi)=1$;
\item $\cld_G(\pi)=1$;
\item $\pi$ is a non-trivial free group with basis a $G$-set.
\ee
\end{thm}

\begin{proof}
The implications (1)$\implies$(2)$\implies$(3) follow from Theorem \ref{EqEG} and the fact that all three invariants are zero if and only if $\pi$ is trivial.

Let us prove that (4)$\implies$(1). Suppose that $\pi$ is a non-trivial free group with basis a $G$-set. Then as a $K(\pi,1)$ we may take a graph $X$ with a single vertex and edges indexed by the basis elements. The group $G$ acts by fixing the vertex and permuting the edges according to the action of $G$ on the basis (preserving orientations), turning $X$ into a $G$-CW complex. The universal cover $\widetilde{X}$ is a tree, with vertices indexed by the elements of $\pi$. It is also a $(\GG)$-CW complex, which we claim is a model for $E_\FG(\GG)$. The action of $\GG$ on the vertices is given by $(\alpha,g)\cdot \alpha_0 = \alpha{}^g \alpha_0$, so as in the proof of Lemma \ref{Epimodel} we conclude that the isotropy of any vertex of $\widetilde{X}$ is in $\FG$. Since an element fixes an edge if and only if it fixes both its vertices, and $\FG$ is closed under intersections, we see that $\widetilde{X}$ has isotropy in $\FG$. Finally we observe that since $G$ fixes the vertex $1\in \widetilde{X}$, any conjugate of $G$ must fix a vertex and hence the fixed sub-complexes $\widetilde{X}^H$ for $H\in \FG$ are all non-empty and are trees, therefore contractible. We conclude that $\gld_G(\pi)=1$.

It remains only to prove that (3)$\implies$(4). So suppose that $\cld_G(\pi)=1$. By Shapiro's Lemma $\cld(\pi)\leq\cld_G(\pi)=1$, and hence $\pi$ is a non-trivial free group by Stallings--Swan. It remains to show that $\pi$ admits a basis which is permuted by $G$.

Firstly we claim that $\FG=\mathcal{FIN}$. This is a consequence of our main algebraic result Theorem \ref{mainalg}. For suppose $H$ is a finite subgroup of $\GG$ not in $\FG$. Then $H\cap\FG$ is a family of proper subgroups of $H$, and Theorem \ref{mainalg} together with Shapiro's Lemma yields
\[
2\le \cld_{H\cap\mathcal{G}}(H)\leq \cld_\FG(\GG)=\cld_G(\pi),
\]
contradicting $\cld_G(\pi)=1$.

Thus we find that $\cld_G(\pi)=\underline\cld(\pi\rtimes G)=1$; but then a well-known result of Dunwoody
\cite{Dunwoody,DicksDunwoody} implies that $\gld_G(\pi)=\underline\gld(\pi\rtimes G) =1$. Hence $\GG$ acts on a tree $T$ with finite stabilisers. (This also follows from a result of Karrass--Pietrowski--Solitar \cite{KPS}, Cohen \cite{Cohen} and Scott \cite{Scott}, since $\GG$ is virtually free.) The action of $\pi=\pi\times 1\nsub \GG$ on $T$ is free, and the quotient $X:=T/\pi$ is a $1$-dimensional $G$-CW complex with $\pi$ as fundamental group. Taking the barycentric subdivision if necessary, we may assume that $X$ is a simplicial $G$-graph. The result will follow if we can show that $X$ has a $G$-invariant spanning tree $X_0$, for then the quotient graph $X/X_0$ is a $G$-CW complex model for $K(\pi,1)$ with a single $0$-cell, and the $G$-set of (oriented) $1$-cells gives a basis of $\pi$.

The following lemma is proved in \cite{KanoSakamoto}, under the assumption that $X$ is finite. The same proof can be seen to work for $X$ infinite.

\begin{lem}
Let $G$ be a finite group and let $X=(V,E)$ be a simplicial $G$-graph. If $V^G\neq \varnothing$, then $X$ admits a $G$-invariant spanning tree $X_0$ if, and only if, for each $v\in V$ the sub-graph $X^{G_v}$ fixed by the stabiliser $G_v$ of $v$ is connected.
\end{lem}

Since the action of the finite group $G$ on the tree $T$ must fix some vertex $x_0$, the induced action of $G$ on $X:=T/\pi$ fixes the vertex $v_0=[x_0]$. Thus it suffices to show that for any vertex $v=[x]$ of $X$, the fixed sub-graph $X^{G_v}$ is connected.

 Let $H=G_v$ be the stabiliser of $v=[x]$. Observe that $H$ acts on the orbit $\pi x\subseteq T$. Thus for all $h\in H$ there is a unique $\alpha_h\in \pi$ such that ${}^h x = \alpha_h x$. The function $\varphi:H\to \pi$ given by $h \mapsto \alpha_h^{-1}$ is a $1$-cocycle. For if $h,k\in H$ then
 \begin{align*}
\alpha_{hk} x  & =  {}^{hk} x \\
                    & = {}^{h}\left({}^k x\right)\\
                    & = {}^h\left( \alpha_k x\right)\\
                    & = {}^h \alpha_k {}^h x \\
                    & =  {}^h \alpha_k \alpha_h x,
                    \end{align*}
from which it follows that $\varphi(hk) = (\alpha_{hk})^{-1} = (\alpha_h^{-1}) {}^h(\alpha_k^{-1}) = \varphi(h) {}^h \varphi(k)$.

 Since $\FG=\FIN$, by Proposition \ref{nonabelian} we have $H^1(H;\pi)=\{1\}$ and hence $\varphi$ is principal. This means there exists $\alpha\in \pi$ such that $\alpha {}^h(\alpha^{-1})=\alpha_h^{-1}$ for all $h\in H$. It follows that $\alpha^{-1}x \in \pi x$ is an $H$-fixed point, since for all $h\in H$ we have
 \begin{align*}
 {}^h( \alpha^{-1}x) & = {}^h(\alpha^{-1}){}^h x \\
  & = {}^h(\alpha^{-1}) \alpha_h x \\
  & = \alpha^{-1} \alpha_h^{-1} \alpha_h x \\
  & = \alpha^{-1} x.
  \end{align*}

 Then the unique geodesic in $T$ from $\alpha^{-1}x$ to $x_0$ is contained in $T^H$, and its image in the quotient graph is a path in $X^H$ from $v$ to $v_0$, which shows that $X^H$ is connected. This completes the proof.
\end{proof}

\section{Proof of Theorem \ref{mainalg}}

In this section we give the proof of Theorem \ref{mainalg}, which states that for any finite group $\G$ and any proper family $\F$ we have $\cld_\F(\G)\ge2$. We begin with two lemmas which reduce to the case of finite simple groups and the family $\PP$ of all proper subgroups.

\begin{lem} \label{lemma: proper family} If Theorem \ref{mainalg} holds for any finite group $\G$ and the family $\PP$ of \emph{all} proper subgroups of $\G$, then it holds for any finite group $\G$ and any family $\F$ of proper subgroups of $\G$.
\end{lem}

\begin{proof} If $\F=\PP$, then we are done. Assume that $\F \neq \PP$. Then we can choose $H_1$ which is a proper subgroup and which is not an element of $\F$. By Lemma \ref{lemma: Shapiro},
\[\cld_{H_1 \cap \F}(H_1) \leq \cld_\F(\G).\]
The family $H_1 \cap \F$ only contains proper subgroups. If it contains all proper subgroups, then we are done by assumption and the latter inequality. Otherwise choose a proper subgroup $H_2 \leq H_1$ which does not belong to $H_1 \cap \F$. We can continue this procedure inductively. Since the group $\G$ is finite, the procedure has to terminate after finitely many steps, meaning that eventually we will find a subgroup $H$ such that $H \cap \F$ is the family of all proper subgroups of $H$. Now again the assumption and Lemma \ref{lemma: Shapiro} imply
\[ \cld_\F(\G) \geq \cld_{H \cap \F}(H) \geq 2.\]
\end{proof}

Next we reduce the proof of Theorem \ref{mainalg} to simple groups. For this consider a  group $\G$, a family of subgroups $\F$, and a normal subgroup $N\nsub \G$. Consider the family of subgroups of $\G/N$ defined by
\[\F_N := \{L \leq \G/N \mid p^{-1}(L) \in \F \} \]
where $p \colon \G \to \G/N$ is the projection. This map induces a functor
\[p^* \colon \mathcal{O}_{\F_N}\G/N \to \OFG\]
which pulls back the group action. More precisely it sends $(\G/N)/(S/N)$, where $N \leq S$ and $S \in \F$, to the coset $\G/S$. The latter functor in turn by pre-composition provides a functor
\[p_* \colon \OFG-\text{Mod} \to \mathcal{O}_{\F_N}\G/N-\text{Mod}.\]
Since it is induced by pre-composition, $p_*(\ul{\Z})=\ul{\Z}$ and $p_*$ is exact and it preserves direct sums. Finally, for any $S \in \F$, we have an isomorphism of $\mathcal{O}_{\F_N}\G/N$-modules
\[ p_*(\Z[\OFG(-, \G/S)]=\Z[\OFG(p^*(-), \G/S)] \cong  \Z[\mathcal{O}_{\F_N}\G/N(-, (\G/N)/(S/N)],\]
if $N \leq S$, and
$p_*(\Z[\OFG(-, \G/S)]=0,$
if $N$ is not contained in $S$. These follow since $(\G/S)^N$ is isomorphic as a $\G/N$-set to $(\G/N)/(S/N)$ when $N \leq S$ and is empty otherwise. Hence we conclude that $p_*$ preserves projective resolutions of $\ul{\Z}$ and hence the following well-known lemma holds:

\begin{lem} \label{lemma: quotient cld} Let $\G$ be a group, $\F$ a family of subgroups of $\G$, and $N\nsub\G$ a normal subgroup. Then
\[ \cld_{\F_N}(\G/N) \leq \cld_{\F}(\G). \]
\end{lem}

As a consequence we get a further reduction for the proof of Theorem \ref{mainalg}:

\begin{cor} \label{corollary: simple groups and proper family} If Theorem \ref{mainalg} holds for any finite simple group $G$ and the family $\PP$ of all proper subgroups of $G$, then it holds for any finite group $\G$ and any family $\F$ of proper subgroups.
\end{cor}

\begin{proof} Let $\Gamma$ be any finite group and $\PP_{\Gamma}$ the family of all proper subgroups of $\Gamma$. By Lemma \ref{lemma: proper family} it suffices to prove that $\cld_{\PP_{\Gamma}}(\Gamma) \geq 2$. If $\Gamma$ is simple we are done by the assumption. If it is not simple, then there exists a proper non-trivial normal subgroup $N_1 \leq \Gamma$. By Lemma \ref{lemma: quotient cld}, we get
\[ \cld_{\PP_{\Gamma/N_1}} (\Gamma/N_1) \leq \cld_{\PP_{\Gamma}}(\Gamma),\]
where $\PP_{\Gamma/N_1}$ is the family of all proper subgroups of $\Gamma/N_1$. Now if $\Gamma/N_1$ is simple, then we are done. Otherwise we find a proper non-trivial normal subgroup $N_2$ in $\Gamma/N_1$. We can continue the procedure inductively. Since the group $\Gamma$ is finite and each step produces a group of strictly smaller cardinality than in the previous step, this procedure has to terminate after finitely many steps, meaning that we will find a simple quotient $G$ of $\Gamma$ such that
\[\cld_{\PP}(G) \leq \cld_{\PP_{\Gamma}}(\Gamma),\]
where $\PP$ is the family of all proper subgroups of $G$. By assumption this finishes the proof.
\end{proof}

The rest of the proof consists in showing that $\cld_\PP(G)\ge2$ for any finite simple group $G$ and the family $\PP$ of all proper subgroups.

First we give a general lemma which bounds $\cld_\F(\G)$ from below in terms of the cohomology of the poset of subgroups in $\F$.

\begin{lem} \label{lemma: poset dimension vs orbit dimension} Let $\G$ be a group  and $\F$ a family of subgroups of $\G$. Let $\mathcal{A}_{\F}(\G)$ denote the poset of all subgroups of $\G$ contained in $\F$. Then
\[\cld(\mathcal{A}_{\F}(\G)) \leq \cld_{\F}(\G).\]
\end{lem}

\begin{proof} Consider the category $\OFG_*$ of pointed objects in $\OFG$. The objects of $\OFG_*$ are pairs $(\G/H, \gamma H)$, where $H \in \F$ and morphisms are equivariant maps which respect the distinguished cosets (note that $\OFG_*$ is the Grothendieck construction of the functor $\OFG \to \text{Sets}$, sending a coset to its underlying set). It is easy to see that there is at most one morphism between any two objects in $\OFG_*$.

We have the forgetful functor
\[u \colon \OFG_* \to   \OFG\]
which forgets the distinguished coset. Pre-composing with $u$ induces a functor
\[u^* \colon \OFG-\text{Mod} \to \OFG_*-\text{Mod}\]
which clearly sends $\ul{\Z}$ to $\ul{\Z}$, is exact and preserves direct sums. Moreover, for any $H \in \F$, we have an isomorphism of $\OFG_*$-modules
\[u^*(\Z[\OFG(-, \G/H)]) \cong \bigoplus_{\gamma H \in \G/H} \Z[\OFG_*(-, (\G/H, \gamma H))]. \]
Hence $u^*$ preserves projective resolutions of $\ul{\Z}$. Consequently,
\[ \cld(\OFG_*) \leq \cld_{\F}(\G)\]
Now the obvious functor $\mathcal{A}_{\F}(\G) \to \OFG_*$ sending $H \in \F$ to $(\G/H,1H)$ is fully-faithful and essentially surjective, showing that $\mathcal{A}_{\F}(\G)$ is equivalent to $\OFG_*$ and thus $\cld(\OFG_*)=\cld(\mathcal{A}_{\F}(\G))$. This finishes the proof.
\end{proof}

In view of Corollary \ref{corollary: simple groups and proper family} and Lemma \ref{lemma: poset dimension vs orbit dimension}, Theorem \ref{mainalg} will be proved once we can prove the following result.

\begin{prop} \label{proposition: crowns} Let $G$ be a non-abelian finite simple group and $\PP$ denote the family of all proper subgroups of $G$. Then $\cld(\A_{\PP}(G)) \geq 2$.
\end{prop}
\begin{rem} If $G$ is a cyclic group of prime order then $\cld_{\PP}(G)=\cld(G)=\infty$, and in that case it trivially holds that the cohomological dimension is bigger than 1.
\end{rem}
To prove the proposition, we will begin by proving the following:
\begin{prop}\label{proposition: subgroups}
Let $G$ be a non-abelian finite simple group. The lattice $\A_{\PP}(G)$ of proper subgroups of $G$ contains two non-empty collections of subgroups $A$ and $B$ such that:
\begin{enumerate}
\item All subgroups in $A$ are maximal.
\item For every $b\in B$ there are at least two subgroups $a_1,a_2\in A$ such that $b\subseteq a_1$ and $b\subseteq a_2$.
\item For every $a\in A$ there are at least two subgroups $b_1,b_2\in B$ such that $b_1\subseteq a$ and $b_2\subseteq a$.
\item The cardinality of every $b\in B$ is the maximal cardinality of intersection of two maximal subgroups.
\end{enumerate}
%full sublattice of the form
%$$\xymatrix{
%H_1 \ar@{-}[d]\ar@{-}[drrr]& H_2\ar@{-}[d]\ar@{-}[dl] & \cdots & H_n\ar@{-}[dl]\ar@{-}[d] \\
%T_1 & T_2\ar@{-}[ur] & \cdots & T_n
%}$$ for some $n\geq 2$, where all $H_i$ are maximal subgroups of $G$, and all the $T_i$ subgroups are conjugate.
\end{prop}
\begin{proof}
We begin by showing that $G$ contains two maximal subgroups $H$ and $K$ such that $H\cap K$ is non-trivial and not normal in $H$ nor in $K$.  For this, start with a maximal subgroup $H<G$. Since $H$ is maximal, it is self normalizing, so $H\neq H^g$ for every $g\notin H$.
Now, if $H\cap H^g=1$ for every $g\notin H$ then $H$ is called a \textit{Frobenius complement} in  $G$. By \cite[Theorem 7.2]{Isaacs}, $G$ has a normal subgroup $N$ such that $HN=G$ and $H\cap N=1$. This contradicts the simplicity of $G$.

We thus know that some maximal subgroups of $G$ intersect non-trivially.
Take such a pair of maximal subgroups $(H_1,H_2)$ for which $|H_1\cap H_2|$ is maximal. We first would like to show that we can assume that $H_1\cap H_2$ is not normal in $H_1$ nor in $H_2$. Indeed, it is impossible that $H_1\cap H_2$ is normal in both $H_1$ and $H_2$, because then the normalizer of $H_1\cap H_2$ contains $\langle H_1,H_2\rangle = G$ since $H_1$ and $H_2$ are maximal and distinct. This implies that $H_1\cap H_2$ is normal in $G$, contradicting the fact that $G$ is simple.

Assume then that $H_1\cap H_2$ is normal in $H_1$ but not in $H_2$. Take $x\in H_1\backslash H_2$. Then $x(H_1\cap H_2)x^{-1}=H_1\cap H_2$. In particular we have that $H_1\cap H_2 = xH_1x^{-1}\cap xH_2x^{-1}\subseteq xH_2x^{-1}$.
The maximality of $H_2$ implies that $N_G(H_2)=H_2$ because $G$ is simple, so $H_2\neq xH_2x^{-1}$. The subgroup $H_1\cap H_2$ is then contained in the two distinct maximal subgroups $H_2$ and $xH_2x^{-1}$. By the maximality assumption on $|H_1\cap H_2|$ we see that $H_1\cap H_2 = H_2\cap xH_2x^{-1}$. This intersection is not normal in $H_2$, and by conjugating by $x$ we see that it is also not normal in $xH_2x^{-1}$ as required.

We thus have a diagram of the form
$$
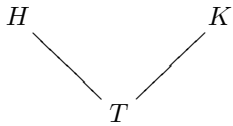
\xymatrix{H \ar@{-}[rd] & & K\ar@{-}[ld]\\ & T & }
$$
where $H$ and $K$ are maximal subgroups of $G$ and $T=H\cap K$ is non-trivial and not normal in $H$ nor in $K$. Since $T$ is not normal in $K$ there is $y_1\in K\backslash H $ such that $y_1Ty_1^{-1}\neq T$. There is also $y_2\in H\backslash K$ such that $y_2Ty_2^{-1}\neq T$.  The previous diagram then gives us the following diagram:
$$
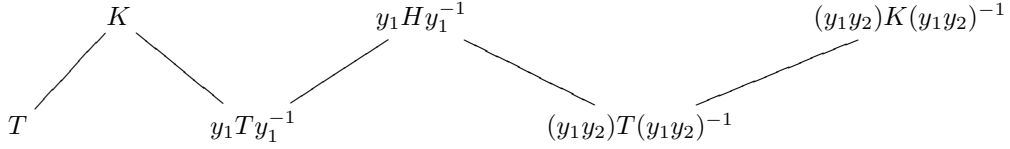
\xymatrix{& K \ar@{-}[rd] & & y_1Hy_1^{-1}\ar@{-}[ld]\ar@{-}[rd] & & (y_1y_2)K(y_1y_2)^{-1}\\T\ar@{-}[ur]&  & y_1Ty_1^{-1} & &  (y_1y_2)T(y_1y_2)^{-1}\ar@{-}[ru]&  }
$$

Take now the subgroup collections $$A:=\{(y_1y_2)^nK(y_1y_2)^{-n}\}_{n\in \Z}\cup \{ (y_1y_2)^ny_1H((y_1y_2)^ny_1)^{-1}\}\text{ and }$$
$$B:=\{(y_1y_2)^nT(y_1y_2)^{-n}\}_{n\in \Z}\cup \{((y_1y_2)^ny_1)T((y_1y_2)^ny_1)^{-1}\}_{n\in \Z}.$$
We prove that they satisfy the conditions of the proposition.
Firstly, all subgroups in $A$ are maximal, since they are all conjugate to either $H$ or to $K$.
Secondly, in order to prove the second and third conditions it will be enough to prove them for the subgroups $K$ and $y_1Hy_1^{-1}$ in $A$ and $T$ and $y_1Ty_1^{-1}$ in $B$, since all other subgroups are conjugate to these subgroups by $(y_1y_2)^n$ for some $n\in \Z$.

The group $K$ contains $T$ and $y_1Ty_1^{-1}$. These are different subgroups in $B$ by the assumption on $y_1$. Similarly the group $y_1Hy_1^{-1}$ contains the subgroups $y_1Ty_1^{-1}$ and $(y_1y_2)T(y_1y_2)^{-1}$ from $B$. Again, these subgroups are distinct because of the way we chose $y_2$.

The subgroup $y_1Ty_1^{-1}$ is contained in $K$ and in $y_1Hy_1^{-1}$. These two subgroups are different, since if $y_1Hy_1^{-1}=K$, then from the fact that $y_1\in K$ it follows that $H=K$ which is a contradiction.
Simiarly, $T$ is contained in $(y_1y_2)^{-1}y_1H((y_1y_2)^{-1}y_1)^{-1} = y_2^{-1}Hy_2=H$ and in $K$. Again, these two subgroups are different. Finally, the last condition on subgroups in $B$ follows from the way we constructed the subgroups in $B$.
This finishes the proof of the proposition.
\end{proof}
We next recall some notations and results from \cite{Cheng}. Since the modules in \cite{Cheng} are covariant functors, and the modules here are contravariant functors, we will change the notations accordingly.
If $\CC$ is a lattice then the \textit{depth} of $x\in \CC$ is the maximal $n$ such that there is a chain of the form $x=x_n<x_{n-1}<\cdots< x_0$ in $\CC$. For elements $x,y\in \CC$ we say that $x$ \textit{covers} $y$ if $y< x$ and there is no $z\in\CC$ such that $y<z<x$.
A vertex in $\CC$ is called \textit{superfluous} if it is either maximal and covers a unique element, or it is of depth 1 and is covered by a unique maximal element.
The poset $E(\CC)$ is the poset resulting from $\CC$ by successively removing superfluous elements from $\CC$. Cheng showed that the isomorphism type of $E(\CC)$ is independent of the order of removal of superfluous elements from $\CC$ \cite[Proposition 1.4]{Cheng}. He also proved the following:
\begin{lem}\label{lemma: cheng}\cite[Lemma 1.7]{Cheng} Let $\CC$ be a finite poset with an initial object. Then $\cld(\CC)\leq 1$ if and only if $E(\CC)=\{*\}$.
\end{lem}
\begin{rem} The lemma in \cite{Cheng} is phrased for finite posets with a terminal object. Since we are considering here contravariant functors instead of covariant functors, we reverse the results accordingly.
\end{rem}
\begin{proof}[Proof of Proposition \ref{proposition: crowns}]
Let $A$ and $B$ be two collections of subgroups of $G$ given by Proposition \ref{proposition: subgroups}. By removing superfluous elements according to a specific regime, we will see that $E(\A_{\PP}(G))$ contains all subgroups in $A$ and in $B$, and is therefore not a singleton. This will be enough by Lemma \ref{lemma: cheng}.

Let $$\CC_0=\A_{\PP}(G)\supseteq \CC_1\supseteq \CC_2\supseteq \cdots\supseteq \CC_n= E(\CC_0).$$
be a chain of posets, where $\CC_{i+1}$ results by removing one superfluous element from $\CC_i$ for every $i$. 	
Since the order of removing the superfluous elements does not change the isomorphism type of $E(\CC_0)$, we can (and we will) assume that
$\CC_{i+1}$ is formed from $\CC_i$ by removing a maximal superfluous element only if there are no superfluous elements of depth 1.

We will prove by induction that $A$ and $B$ are contained in $\CC_i$.
For $i=0$ this is clear. Assume now that $A\cup B\subseteq \CC_i$. Let $x$ be the superfluous element removed from $\CC_i$ to form $\CC_{i+1}$. If $x\notin A\cup B$ we are done. If $x\in A\cup B$ and $x$ has depth 1 in $\CC_i$ then necessarily $x\in B$ since all the elements of $A$ are maximal.
But an element in $B$ of depth 1 is covered by at least two maximal elements in $A$. This implies that $x$ is not superfluous, which is a contradiction.

Assume then that $x$ is a maximal element in $\CC_i$. Then $x\in A$. Since $x$ is superfluous, $x$ covers a unique element $y\in \CC_i$.
There are at least two distinct elements $b_1,b_2\in B$ such that $b_1,b_2\leq x$. Since $y$ is the unique element which $x$ covers, it must hold that $b_1,b_2\leq y$ as well.

We claim that the element $y$ must have depth 1. Indeed, if $y$ is not of depth 1 then there is a chain $y=y_n<y_{n-1}<\cdots<y_0$ where $y_0$ is maximal and $n>1$. Since $x$ covers $y$ it holds that $y_0\neq x$. It follows that $y$ is contained in the intersection $x\cap y_0$. By property (4) of the collection $B$, this implies that the cardinality of $y$ is at most the cardinality of the subgroups in $B$, and since $b_1,b_2\leq y$ we get $b_1=b_2=y$. This contradicts our assumption that $b_1\neq b_2$.

Next, we claim that $y$ is a superfluous element. Since it has depth 1, this means that we formed $\CC_{i+1}$ from $\CC_i$ by removing a maximal superfluous element while $\CC_i$ has a depth 1 superfluous element, contrary to our assumption.

Assume by contradiction that $y$ is not superfluous. Then there is a maximal element $m\neq x$ in $\CC_i$  such that $y\leq m$. In the group $G$ we can thus find a maximal subgroup $m'$ such that $m\leq m'$. We then have the inequality
$$b_1,b_2\leq y\leq x\cap m'.$$
But the cardinality of $b_1$ and $b_2$ is the maximal cardinality among intersection of two different maximal subgroups. This implies that all inequalities are in fact equalities, and we get $b_1=b_2=y=x\cap m'$. But this contradicts the fact that $b_1\neq b_2$, and we are done.
\end{proof}

\begin{ex} It follows from \cite[Example 5.1]{Adem} (see also \cite[Section 6]{BDP}) that $\cld_{\PP}(A_5) \leq 2$, where $\PP$ is the family of all proper subgroups of the alternating group $A_5$. Hence using Theorem \ref{mainalg}, we conclude that $\cld_{\PP}(A_5) = 2$. This example shows that Theorem \ref{mainalg} is optimal.
\end{ex}

\end{document}